\documentclass[a4paper,10pt]{article}
\usepackage{textcomp}
\usepackage{graphicx}
\usepackage{color}
\usepackage[colorlinks]{hyperref}
\usepackage{lscape}
\usepackage[centertags]{amsmath}
\usepackage{amssymb}
\usepackage{amsthm}
\usepackage{newlfont}
\usepackage{amsfonts}
\theoremstyle{plain}
\newtheorem{thm}{Theorem}
\newtheorem{cor}{Corollary}
\newtheorem{lem}{Lemma}
\newtheorem{prop}{Proposition}
\newtheorem{conjec}{Conjecture}
\newtheorem{defn}{Definition}

\newtheorem{prob}{Problem}

\newcommand{\ra}{\rightarrow}

\newcommand{\sub}{\subseteq}
\newcommand{\mo}{$L=v_1v_2...v_n$ }
\newcommand{\nd}{$v\in N^{++}_D(f)$ }
\newcommand{\nnd}{$v\in N^{++}_{D'}(f)$}
\newcommand{\maxp}{There is a maximal directed path $P=m_0y_0\ra ..\ra m_iy_i\ra ...\ra m_ky_k$ in $\Delta$ such that }
\newcommand{\same}{The same order $L$ is a local median order of the obtained tournament $T'$, $f$ is a feed vertex of $L$ and $f$ has the SNP in $T'$}
\newcommand{\lose}{$rs\ra uv $ in $\Delta$, namely $s\ra v$ and $ u\notin N^{++}_D(s)$}

\begin{document}

\begin{center}
\Large \textbf{The Second Neighborhood Conjecture for Oriented Graphs Missing Combs}
\end{center}
\begin{center}
Salman GHAZAL\footnote{\noindent Department of Mathematics, Faculty of Sciences I, Lebanese University, Hadath, Beirut, Lebanon.\\
                       E-mail: salmanghazal@hotmail.com\\

                       }
\end{center}
\vskip1cm
\begin{abstract}
 Seymour's Second Neighborhood Conjecture asserts that every oriented graph has a vertex whose first out-neighborhood is at most as large as its second out-neighborhood. Combs are the graphs having no induced $C_4$, $\overline{C_4}$, $C_5$, chair or $\overline{chair}$.  We characterize combs using dependency digraphs. We characterize the graphs having no induced $C_4$, $\overline{C_4}$, chair or $\overline{chair}$ using dependency digraphs. Then we prove that every oriented graph missing a comb satisfies this conjecture. We then deduce that every oriented comb and every oriented threshold graph satisfies Seymour's conjecture.
\end{abstract}

\begin{section}{Introduction}

\par \hskip0.6cm In this paper, graphs are finite and simple. The vertex set and edge set of a graph $G$ are  denoted by $V(G)$ and $E(G)$ respectively. Two edges of a graph $G$ are said to be adjacent if they have a common endpoint and two vertices $x$ and $y$ are said to be adjacent if $xy$ is an edge of $G$. The \emph{neighborhood} of a vertex $v$ in a graph $G$, denoted by $N_G(v)$, is the set of all vertices adjacent to $v$ and its \emph{degree} is $d_G(v)=|N_G(v)|$. We omit the subscript if the graph is clear from the context. For two set of vertices $U$ and $W$ of a graph $G$, let $E[U, W]$ denote the set of all edges in the graph $G$ that joins a vertex in $U$ to a vertex in $W$. A graph is empty if it has no edges. For $A\sub V(G)$, $G[A]$ denotes the sub-graph  of $G$ induced by $A$. If $G[A]$ is an empty graph, then $A$ is called a stable. While, if $G[A]$ is a complete graph, then $A$ is called a clique set, that is any two distinct vertices in $A$ are adjacent. The complement graph of $G$ is denoted by $\overline{G}$ and defined as follows: $V(G)=V(\overline{G})$ and $xy\in E(\overline{G})$ if and only if $xy\notin E(G)$. \\

\par Directed graphs (digraphs) contains neither loops nor parallel arcs and oriented graphs are orientations of graphs so they are digraphs without digons (directed cycles of length 2). The vertex set and arc set of a digraph $D$ are  denoted by $V(D)$ and $E(D)$ respectively.
Let $D$ denote a digraph and $x, y\in V(D)$. If $(x,y)\in E(D)$, then $y$ is an out-neighbor of $x$, $x$ is an in-neighbor of $y$ and $x$ and $y$ are adjacent. For $v\in V(D)$,  $N^{+}_D(v)$ (resp. $N^{-}_D(v)$)
denotes the (first) out-neighborhood (resp. in-neighborhood), which is the set of all out-neighbors (resp. out-neighbors) of $v$. Whereas, $N^{++}_D(v)$ denotes the\emph{ second out-neighborhood} of $v$, which is the set of vertices that are at distance 2 from $v$, that is, the set of the out-neighbors of the out-neighbors of $v$, excluding the set $N^{+}_D(v)$. The out-degree, in-degree and the second out-degree of $v$ are the following numbers
$d^{+}_D(v):=|N^{+}_D(v)|$, $d^{-}_D(v)=|N^{-}_D(v)|$ and $d^{++}_D(v)=|N^{++}_D(v)|$, respectively.
The minimum out-degree (resp. in-degree) of $D$ is the minimal out-degree (resp. in-degree) of a vertex in $D$. We omit the subscript if the digraph is clear from the context. For short, we write $x\rightarrow y$ if the arc $(x,y)\in E(D)$. Also, we write $x_1\ra x_2\ra ...\ra x_n$, if $x_i\ra x_{i+1}$ for every $0<i<n$. A digraph is empty if it has no arcs. For $A\sub V(D)$, $D[A]$ denotes the sub-digraph of $D$  induced by $A$. \\

We say that $v$ has the \emph{second neighborhood property} (SNP) if $d^{+}(v)\leq d^{++}(v)$. In 1990, P. Seymour conjectured the following statement:

\begin{conjec}
 \textbf{( The Second Neighborhood Conjecture (SNC) \cite{dean})}\\ Every oriented graph has a vertex with the SNP.
\end{conjec}

SNC, if true, will establish a weakening of an important special case of the Caccetta-Haggkvist conjecture, proposed in 1978:\\

\begin{conjec}
 \textbf{ (\cite{CH})}\\ Every digraph $D$ with minimum out-degree at least $|V(D)|/k$, has a directed cycle of length at most $k$.
\end{conjec}

The weakening requires both, the minimum in-degree and the minimum out-degree are at least $|V(D)|/k$  and the particular case is $k=3$. This is still open problem.\\

 In 1996, Fisher \cite{fisher}
proved the SNC for tournaments (orientations of complete graph). Another short proof of Dean's conjecture was given by
Havet and Thomass\'{e} \cite{m.o.}, in 2000, using a tool called (local) median order. In 2007, Fidler and Yuster \cite{fidler} used also median orders to
prove SNC for tournaments missing a matching, using local median orders and dependency digraphs. In 2012, Ghazal proved the weighted version of SNC for tournaments missing a generalized star \cite{a} and in 2013, he proved the SNC for tournaments minus cycle of length 4 or 5. In 2015, Ghazal \cite{ghazal3} refined the result of \cite{fidler} to show that for every tournament missing a matching, there is a certain "feed vertex" having the SNP. \\
\end{section}

\begin{section}{Preliminary definitions and theorems}

\hskip0.6cm A cycle on $n$ vertices is denoted by $C_n=v_1v_2...v_nv_1$ while a path on $n$ vertices is denoted by $P_n=v_1v_2...v_n$. A chair is any graph on 5  distinct vertices $x,y,z,t,v$ with exactly 5 edges $xy,yz,zt$ and $zv$. The co-chair or $\overline{chair}$ is the complement of a chair (see the below figure).\\

A graph $H$ is called forbidden subgraph of a graph $G$ if $H$ is not (isomorphic to) an induced subgraph of $G$.

\begin{center}
\includegraphics[width=8cm, height=7cm]{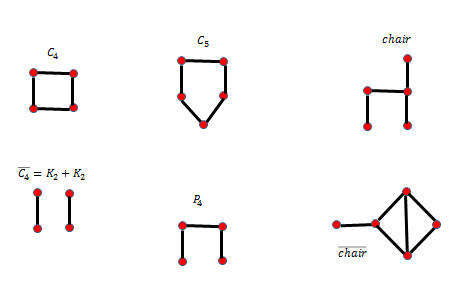}
\end{center}

\begin{defn}
\hskip0.6cm A graph $G$ is a called a split graph if its vertex set is the disjoint union of a stable set $S$ and a clique set $K$. In this case, $G$ is called an $\{S$, $K\}$-split graph.
\end{defn}

If $G$ is an $\{S$, $K\}$-split graph and $\forall s\in S$, $\forall x\in K$ we have $sx\in E(G)$, then $G$ is called a complete split graph.\\

If $G$ is an $\{S$, $K\}$-split graph and $E[S,K]$ forms a perfect matching of $G$, then $G$ is called a perfect split graph.

\begin{defn} (\cite{thresholdch}, \cite{threshold})
A threshold graph $G$ can be defined as follows:
\begin{description}
\item[1) ] $V(G)=\displaystyle\bigcup_{i=1}^{n+1}(X_i\cup A_{i-1})$, where the $A_i$'s and $X_i$'s are pair-wisely disjoint sets.
\item[2) ] $K:=\displaystyle\bigcup_{i=1}^{n+1}X_i$ is a clique and the $X_i$'s are nonempty, except possibly $X_{n+1}$.
\item[3) ] $S:=\displaystyle\bigcup_{i=0}^{n}A_{i}$ is a stable set and the $A_i$'s are nonempty, except possibly $A_0$.
\item[4) ] $\forall 1\leq j\leq i\leq n$, $G[A_i\cup X_j]$ is a complete split graph.
\item[5)] The only edges of $G$ are the edges of the subgraphs mentioned above.
\end{description}

In this case, $G$ is called an $\{S,$ $K\}$-threshold graph.

\end{defn}

\begin{thm}(Hammer and Chv\`{a}tal \cite{thresholdch}, \cite{threshold})
$G$ is a threshold graph if and only if $C_4$, $\overline{C}_4$ and $P_4$ are forbidden subgraphs of $G$.
\end{thm}

\begin{defn}(\cite{combch}) \label{gcdef}
A graph $G$ is called a comb if:
\begin{description}
  \item[1)] $V(G)$ is disjoint union of sets $A_0,...,A_n, M_1,...,M_l,X_1,....,X_{n+1}, Y_2,...,Y_{l+2}$. Let $Y_1=X_1$ (These sets are called the sets of the comb $G$).
  \item[2)] $S:=A\cup M$ is a stable set, where $M=\displaystyle\bigcup_{i=1}^{l}M_{i}$ and $A=\displaystyle\bigcup_{i=0}^{n}A_{i}$
  \item[3)] $K:=X\cup Y$ is a clique, where $X=\displaystyle\bigcup_{i=1}^{n+1}X_{i}$ and $Y=\displaystyle\bigcup_{i=1}^{l+2}Y_{i}$.
  \item[4)] $\forall 1\leq j\leq i\leq n$, $G[A_i\cup X_j]$ is a complete split graph.
  \item[5)]$G[A\cup Y]$ is a complete split graph.
  \item[6)]$\forall 1\leq i\leq l$, $G[Y_i\cup M_i]$ is a perfect split graph.
  \item[7)] $\forall 1\leq i <j \leq l$, $G[Y_j\cup M_i]$ is a complete split graph.
  \item[8)] $\exists 1\leq k_0\leq l$, $\forall i\leq k_0$, $G[Y_{l+1}\cup M_i]$ is a complete split graph.
  \item[9)] $X_{n+1}, Y_{l+2}, Y_{l+1}, M_l$ and $A_0$ are the only possibly empty sets.
  \item[10)] The only edges of $G$ are the edges of the subgraphs mentioned above.
\end{description}

In this case, we say that $G$ is an $\{S,$ $K\}$-comb.
\end{defn}

\begin{thm}\cite{combch}
$G$ is a comb if and only if $C_4$, $\overline{C}_4$, $C_5$, chair and co-chair are forbidden subgraphs of $G$.
\end{thm}

\begin{cor}\cite{combch}
  Every threshold graph is a comb.
\end{cor}

\begin{cor}\cite{combch}
  $G$ is a comb if and only if $\overline{G}$ is a comb.
\end{cor}

\begin{cor}\cite{combch}
  $G$ is a comb if and only if every induced subgraph of $G$ is a comb.
\end{cor}

\begin{prop} \label{gs in gc}
  Let $G$ be a comb defined as in Definition \ref{gcdef}. Then $G'=G-\bigcup_{1\leq i\leq l}E[Y_i, M_i]$ is a threshold graph.
\end{prop}

\begin{proof}

It is clear that $G'$ contains no induced $C_4$, $\overline{C}_4$ or $P_4$. Thus $G'$ is a threshold graph.

\end{proof}

\begin{thm}(\cite{combch}) \label{allowedC5}
  $C_4$, $\overline{C_4}$, chair and co-chair are forbidden subgraphs of a graph $G$ if and only if $V(G)$ is disjoint union of three sets $S$, $K$ and $C$ such that:

  \begin{description}
    \item[1)] $G[S\cup K]$ is an $\{S, K\}$-comb;
    \item[2)] $G[C]$ is empty or isomorphic to the cycle $C_5$;
    \item[3)] every vertex in $C$ is adjacent to every vertex in $K$ but to no vertex in $S$.
  \end{description}

\end{thm}

Let $L=v_1v_2...v_n$ be an ordering of the vertices of a digraph $D$. $L$ is called a local median order of $D$ if it satisfies the \emph{feedback property}: For all $1\leq i\leq j\leq n:$
$$ d^{+}_{[i,j]}(v_i)  \geq  d^{-}_{[i,j]}(v_i) $$
and
$$ d^{-}_{[i,j]}(v_j) \geq  d^{+}_{[i,j]}(v_j)  $$
where $[i,j]:=D[\{v_i,v_{i+1}, ...,v_j\}]$.\\
In this case, the last vertex $v_n$ is called a feed vertex.
A local median order always exist. In fact, any order $L=v_1v_2...v_n$ that maximizes the set of arcs $(v_i,v_j)\in E(D)$ with $i<j$, is a local median order. We will need the following proposition.

\begin{prop}
Suppose that \mo is a local median order of a digraph $D$ and $e=(v_j,v_i)\in D$ with $i<j$.
Then $L$ is a local median order of the digraph $D'$ obtained from $D$ by reversing the orientation
of $e$.
\end{prop}

\begin{proof}
Reversing an arc $(x,y)$ means removing it and adding $(y,x)$.
It is enough to note that reversing the orientation of such an arc with respect to $L$ preserves and strengthens the feedback property of $L$.
\end{proof}

We will use the following theorem.

\begin{thm}\cite{m.o.}
 Every feed vertex of a tournament has the SNP.
\end{thm}

\end{section}

\begin{section}{Characterization using dependency digraphs}

\hskip0.6cm Let $D$ be an oriented graph. For two vertices $x$ and $y$, we call $xy$ a \emph{missing edge} if $(x,y) \notin E$ and $(y,x)\notin E$. A vertex $v$ is \emph{whole} if it is not incident to any missing edge, i.e., $N^{+}(v)\cup N^{-}(v)=V(D)-\{v\}$.

 The \emph{missing graph} $G$ of $D$ is the graph formed by the missing edges, formally, $E(G)$ is the set of all the missing edges and $V(G)$ is the set of the non-whole vertices (vertices incident to some missing edges). In this case, we say that $D$ is \emph{missing} $G$.

\par We say that a missing edge $x_1y_1$ \emph{loses to} a missing edge $x_2y_2$ if:
$x_1\rightarrow x_2$, $y_2\notin N^{+}(x_1)\cup N^{++}(x_1)$, $y_1\rightarrow y_2$ and $x_2\notin N^{+}(y_1)\cup N^{++}(y_1)$.

The \emph{dependency digraph} $\Delta_{D}$ (or simply $\Delta$) of $D$ is defined as follows: Its vertex set consists of all the missing
edges of $D$ and $(ab,cd)\in E(\Delta)$ if and only if $ab$ loses to $cd$. Note that $\Delta$ may contain digons.

These digraphs were used in \cite{fidler, contrib} to prove SNC for some oriented graphs.\\

\begin{defn}\cite{a}
A missing edge $ab$ is called \emph{good} if:\\
$(i)$   $(\forall v \in V\backslash\{a,b\})[(v\rightarrow a)\Rightarrow(b\in N^{+}(v)\cup N^{++}(v))]$ or\\
$(ii)$ $(\forall v \in V\backslash\{a,b\})[(v\rightarrow b)\Rightarrow(a\in N^{+}(v)\cup N^{++}(v))]$.\\
If $ab$ satisfies $(i)$ we say that $(a,b)$ is a \emph{convenient orientation} of $ab$.\\
If $ab$ satisfies $(ii)$ we say that $(b,a)$ is a \emph{convenient orientation} of $ab$.\\
\end{defn}

\begin{lem}\cite{contrib}\label{goodmissinedgelemma}
 Let $D$ be an oriented graph and let $\Delta$ denote its dependency digraph. A missing edge $ab$ is good
if and only if its in-degree in $\Delta$ is zero.
\end{lem}

In \cite{a}, threshold graphs are characterized using dependency digraphs and are called generalized stars.

\begin{thm}(\cite{a})
Let $G$ be a graph. The following statements are equivalent:
\begin{description}
  \item[i)] $G$ is a threshold graph;
  \item[ii)] Every missing edge of every oriented graph missing $G$ is good;
  \item[iii)]  The dependency digraph of every oriented graph missing $G$ is empty.
\end{description}
\end{thm}

\begin{prob}
Let $\mathcal{H}$ be a family of digraphs (digons are allowed) and let
$\mathcal{F}(\mathcal{H})=\{G$ is a graph; $\forall D$ missing $G$, $\Delta_{D}\in \mathcal{H}\}$ .
Characterize $\mathcal{F}(\mathcal{H})$.
\end{prob}

\begin{prop}\label{inducedforce}
  Suppose that $G\in \mathcal{F}(\mathcal{H})$. If $G'$ is an induced subgraph of $G$, then $G'\in \mathcal{F}(\mathcal{H^*})$, where $\mathcal{H^*}=\{H^*; \exists H \in \mathcal{H}, H=H^*$ plus a set of  isolated vertices of $H\}$.
\end{prop}

\begin{proof}
Suppose that $G\in \mathcal{F}(\mathcal{H})$ and assume first that $G'=G-v$, for some $v\in V(G)$. Let $D'$ be any oriented graph missing $G'$. Let $\alpha$ and $\beta$ be 2 distinct extra vertices neither in $D'$ nor in $G$.  Define $D$ as follows. The missing graph of $D$ is $G$ and $V(D)=V(D')\cup \{v,\alpha,\beta\}$. $D-\{v,\alpha,\beta\}=D'$. The arcs $(\alpha, v)$, $(v, \beta)$ and $(\alpha, \beta)$ are in $D$. For every $x\in V(D')$, if $xv\notin E(G)$, then $(x,v)\in E(D')$. Finally, for every $x\in V(D')$, the arcs $(x,\alpha)$ and $(\beta, x)$ are in $D$. Then the addition, in this way, of $v, \alpha$ and $\beta$ to $D'$ neither affects the losing relations between missing edges of $D'$ nor creates new ones. Hence, $\Delta_{D}$ is equal to $\Delta_{D'}$ plus isolated vertices (these isolated vertices are the edges of $G$ incident to $v$). Since $D$ is missing $G$, then $\Delta_{D}\in \mathcal{H}$. Whence, $\Delta_{D'}\in \mathcal{H^*}$. Thus $G'\in \mathcal{F}(\mathcal{H^*})$. Now, the proof follows by induction on the number of vertices removed from $G$ to obtain the induced subgraph.

\end{proof}

Note that $\mathcal{H}\sub \mathcal{H^*}$. Moreover, it is clear that, if $H^*\in \mathcal{H^*}$ and $ H \in \mathcal{H}$ such that $H=H^*$ plus a set of  isolated vertices of $H$, then $H^*$ is an induced subgraph of $H$ and differs from $H$ only by a set of isolated vertices.

It is obvious that if $\mathcal{A}$ and $\mathcal{B}$ are two sets of digraphs such that $\mathcal{A}\sub \mathcal{B}$, then $\mathcal{F}(\mathcal{A})\sub \mathcal{F}(\mathcal{B})$.\\

\begin{prob}
Let $\vec{\mathcal{P}}$ be a family of all digraphs consisting of vertex disjoint paths only. Characterize $\mathcal{F}(\vec{\mathcal{P}})$.
\end{prob}

\begin{prop}
$G \in \mathcal{F}(\vec{\mathcal{P}})$ if and only if $G' \in \mathcal{F}(\vec{\mathcal{P}})$, for every $G'$ induced subgraph of $G$.
\end{prop}

\begin{proof}

Enough to note that  $\mathcal{F}(\vec{\mathcal{P}})=\mathcal{F}(\vec{\mathcal{P}}^*)$, because every isolated vertex is a directed path.

\end{proof}

\begin{prop}\label{notinFP}
 $\overline{C}_4$, chair and co-chair are not in $\mathcal{F}(\vec{\mathcal{P}})$.
\end{prop}

\begin{proof}
  Let $D$ be the oriented graph with only the vertices $a$, $b$, $c$ and $d$ and only the arcs $(a,c)$, $(b,d)$, $(d,a)$ and $(c,b)$. Then $D$ is missing $\overline{C}_4$, $ab$ loses to $cd$ and $cd$ loses to $ba$. Thus $\Delta_D\notin \vec{\mathcal{P}}$.\\

  Let $D'$ be the oriented graph with only $a$, $b$, $c$, $d$ and $x$ and only the arcs $(a,d)$, $(b,c)$, $(c,a)$, $(b,x)$, $(x,a)$ and $(x,c)$. Then $D'$ is missing a chair, $ab$ loses to both $dc$ and $dx$. Thus $\Delta_{D'}\notin \vec{\mathcal{P}}$.\\

  Let $D''$ be the oriented graph with only $a$, $b$, $c$, $d$ and $x$ and only the arcs $(a,c)$, $(b,d)$, $(d,a)$, $(a,c)$. Then $D''$ is missing a co-chair, $ab$ loses to both $dc$ and $dx$. Thus $\Delta_{D''}\notin \vec{\mathcal{P}}$.
\end{proof}

It is easy to check the following:

\begin{prop}(\cite{contrib})
  $C_4, C_5\in \mathcal{F}(\vec{\mathcal{P}})$.
\end{prop}

\begin{thm}
  Let $G$ be a graph having neither induced $C_4$ nor induced $C_5$. Then $G\in \mathcal{F}(\vec{\mathcal{P}})$ if and only if $G$ is a comb.
\end{thm}

\begin{proof}

\emph{Necessary Condition.} Since $G\in \mathcal{F}(\vec{\mathcal{P}})$, then $\overline{C}_4$, chair  and co-chair are not induced subgraphs of $G$. However, by given $C_4$ and $C_5$ are not induced subgraphs of $G$. Then $G$ is a comb.\\

\emph{Sufficient Condition.} Let $D$ be an oriented graph missing a comb $G$. We follow the previous notations. Using the definition of $G$, each possible losing relation can occur between two edges in $E[Y_t, M_t]$, for some $t$. For $i=1,2,3$ suppose that $a_ix_i\in E[Y_t, M_t]$ with $a_i\in M_t$ and $x_i\in Y_t$. Assume
$a_1x_1$ loses to the two others. Then we have $a_1\rightarrow x_3$, $x_1\rightarrow a_2$, $a_2\notin N^{++}(a_1)\cup N^{+}(a_1)$
and $x_3\notin N^{++}(x_1)\cup N^{+}(x_1)$. Since $a_2x_3$ is not a missing edge (the only missing edge of $D[Y_t\cup M_t]$ incident to $a_2$ is $a_2x_2$), then either $a_2\rightarrow x_3$ or
$a_2\leftarrow x_3$. Whence, either $x_3\in N^{++}(x_1)\cup N^{+}(x_1)$ or $a_2\in N^{++}(a_1)\cup N^{+}(a_1)$. A contradiction. Thus the maximum out-degree in $\Delta$ is 1. Similarly, we can prove that the maximum in-degree in $\Delta$ is 1. Thus $\Delta$ is composed of directed cycles and paths only. \\

Assume that $\Delta$ contains a directed cycle $a_1b_1\ra...\ra a_nb_n\ra a_1b_1$, with the $a_i$'s in $M_t$ and $b_i$'s in $Y_t$, for some $t$. Then we must have $a_{i+1}\ra a_i, \forall i<n$ and $a_1\ra a_n$. We prove that $\forall 1\leq i<n $, $a_i\ra a_n$. It is true for $i=1$. Assume it is true for $i-1$. Then $a_{i-1}\ra a_n$. Since $a_{i-1}b_{i-1}$ loses to $a_ib_i$, then $a_i\notin N^{++}(a_{i-1})$. But $a_ia_n$ is not a missing edge of $D$, then we must have $a_i\ra a_n$, since otherwise $a_{i-1}\ra a_n\ra a_i$ in $D$ which is a contradiction. Thus we have proved it by induction. In particular, $a_{n-1}\ra a_n$, a contradiction. Thus $\Delta$ has no directed cycles. This shows that $G\in \mathcal{F}(\vec{\mathcal{P}})$

\end{proof}

\begin{cor}
The following statements are equivalent:
\begin{description}
\item[i)] $C_4$ is a forbidden subgraph of $G$ and $G\in \mathcal{F}(\vec{\mathcal{P}})$
  \item[ii)] $C_4$, $\overline{C_4}$, chair and co-chair are forbidden subgraphs of a graph $G$
  \item[iii)] $V(G)$ is disjoint union of three sets $S$, $K$ and $C$ such that: \begin{description}
    \item[1)] $G[S\cup K]$ is an $\{S, K\}$-comb;
    \item[2)] $G[C]$ is empty or isomorphic to the cycle $C_5$;
    \item[3)] every vertex in $C$ is adjacent to every vertex in $K$ but to no vertex in $S$.
  \end{description}
\end{description}

\end{cor}

\begin{proof}
$ii)$ follows from $i)$ by Proposition \ref{notinFP} and $iii)$ follows from $ii)$ by Theorem \ref{allowedC5}. Suppose $iii)$ holds. Then $C_4$ is forbidden subgraph of $G$ by the structure of $G$. Since every vertex in $C$ is adjacent to every vertex in $K$, then there is no losing relation between any edge from $G[C]$ and any edge from $G[S\cup K]$. However, $G[S\cup K]$ is a comb, then it is in $\mathcal{F}(\vec{\mathcal{P}})$ and $G[C]$ is empty or isomorphic to the cycle $C_5$, whence it is in $\mathcal{F}(\vec{\mathcal{P}})$. Thus $G\in \mathcal{F}(\vec{\mathcal{P}})$.
\end{proof}

\end{section}

\begin{section}{Second Neighborhood Conjecture}

\begin{thm}
Every oriented graph missing a comb satisfies SNC.
\end{thm}

\begin{proof}
  Let $G$ be a comb as in definition \ref{gcdef}. Let $D$ be an oriented graph missing $G$. Then its dependency digraph $\Delta$ consists of disjoint directed paths only and each of its arc occur only between two edges in the same set $E[Y_j,M_j]$, for some $j$. Let $P=m_0y_0\ra ..\ra m_iy_i\ra ...\ra m_ky_k$  be a maximal directed path in $\Delta$, with $m_i$'s in $ M_j$ and $y_i$'s in $Y_j$. By lemma \ref{goodmissinedgelemma}, $m_0y_0$ is a good missing edge, so it has a convenient orientation. If $(m_0,y_0)$ is a convenient orientation, the add the arcs $m_{2i}y_{2i}$ and the arcs $(y_{2i+1}m_{2i+1})$ to $D$. Else $(y_0,m_0)$ is the convenient orientation in which case we add the arcs $(y_{2i}m_{2i})$ and the arcs $(m_{2i+1})y_{2i+1}$ to $D$. We do this for every maximal directed path of $\Delta$. The obtained oriented graph $D'$ is missing $G'=G-\cup E[Y_j, M_j]$ which is a generalized star by proposition \ref{gs in gc}. We assign to every missing edge of $D'$ (which is good by theorem 9) a convenient orientation and add it to $D'$. The obtained digraph $T$ is a tournament. Let $L$ be a local median order of $T$ and let $f$ denote its feed vertex. Then $f$ has the SNP in $T$. We prove that $f$ has the SNP in $D$. We have many cases.\\

  \noindent \textbf{case 1:} Assume that $f$ is a whole vertex. Then, clearly, $f$ gains no new out-neighbor. Assume $f\ra u\ra v \ra f$ in $T$. Then $f\ra u$ and $v\ra f$ in $D$.\\
  case 1.1: If $u\ra v $ in $E(D)$, then $v\in N^{++}_D(f)$.\\
  case 1.2: If $u\ra v$ in $E(D')-E(D)$, then either $(u,v)$ is a convenient orientation (w.r.t. $D$) and hence \nd or there is \lose . But $f\ra u$, then the non-missing edge $fs$ is oriented as $(f,s)$ in $E(D)$. Thus $f\ra s\ra v$ in $D$.\\
  case 1.3: If $u\ra v$ in $E(T)-E(D')$. Then $(u,v)$ is a convenient orientation w.r.t $D'$. Hence \nnd . But this is discussed in case 1.2.\\

  \noindent \textbf{case 2:} $\exists 1\leq t\leq l$ such that $f\in M_t$. \maxp $f=m_i$.\\

  \noindent \textbf{case 2.1:} Assume $(y_i,m_i)\in E(D')$. Reorient all the missing edges incident to $f$ towards $f$. \same . Clearly, $f$ gains no new first out-neighbor. We prove $f$ gains no new second out-neighbor. Assume $m_i\ra u\ra v\ra m_i$ in $T'$. Then $(m_i, u)\in E(D)$ and $(u,v)in E(T)$.\\

  subcase a: If $u\ra v$ in $E(D)$, then clearly \nd .\\

  subcase b: If $(u,v)$ in $E(D')-E(D)$. Either $(u,v)$ is a convenient orientation (w.r.t. $D$) and hence \nd or there is \lose . There is $j$ such that $rs, uv\in E[Y_j, M_j]$. Assume $m_i=r$. Then $y_i=s$, $u=y_{i+1}$ and $v=m_{i+1}$. Since $(y_i,m_i)\in E(D')$, then $(m_{i+1}, y_{i+1})\in E(D')$, that is $(v,u)\in E(D')$, a contradiction. so $m_i\neq r$. Assume $s=m_i$. Then $u=m_{i+1}$. However $(m_{i+1},m_i)\in E(D)$, then $(u,f)\in E(D)$, a contradiction. So $s\neq m_i$. Now we prove that $m_is$ is not a missing edge. If $v\in Y_j$, then $s\in M_j$. Then $m_is$ is not a missing edge. Else $v\notin Y_j$. Whence, $u\in Y_j$ and $s\in Y_j$. Since $m_iu$ is not a missing edge, then by definition of $G$, $fs=m_is$ is also not missing edge. But $f\ra u$ in $E(D)$ and $u\notin N^{++}_D(s)$, then we must have $f\ra s$ in $E(D)$. Thus $f\ra s \ra v$ in $D$.\\

  subcase c: If $(u,v)\in E(T)-E(D')$. Then $(u,v)$ is a convenient orientation w.r.t $D'$. But $f\ra u$ in $E(D)$ and $E(D')$, then \nnd. But this is already treated in subcase b.\\

  \noindent \textbf{case 2.2:} Assume $(m_i,y_i)\in E(D')$.\\

  \noindent \textbf{case 2.2.1:} Assume $i<k$. Reorient all the missing edges incident to $m_i$ towards $m_i$ except the arc $(m_i,y_i)$. \same . Then $f$ gains only $y_i$ as an out-neighbor. We prove that $f$ gains only $m_{i+1}$ as a second out-neighbor.\\

  subcase a: Suppose $m_i\ra y_i\ra v$ in $T'$ such that $v\neq m_{i+1}$. Then $(y_i, v)\notin E(D')-E(D)$ and $(y_i,v)\in E(T)$.\\

  subcase a.1: Suppose $(y_i,v)\in E(D)$. Since $y_i\in Y_t$, then $y_{i+1}\in Y_t$. Since $y_iv$ is not a missing edge, then by definition of $G$, $y_{i+1}$ is not a missing edge. Since $y_i\ra v$ in $E(D)$ and $y_{i+1}\notin N^{++}_D(y_i)$, then we must have $v\ra y_{i+1}$. Then $m_i\ra y_{i+1}\ra v$ in $D$.\\

  subcase a.2: Suppose $(y_i, v)\in E(T)-E(D')$. Then $(y_i,v)$ is a convenient orientation w.r.t. $D'$. Then $v\in N^{++}_{D'}(m_i)$. Hence there is vertex $u$ such that $m_i\ra u\ra v$ in $D'$.\\

  subcase a.2.1: Suppose $u=y_i$. Then $(y_i,v)\in E(D)$. This case is already treated in subcase a.1.\\

  subcase a.2.2: Suppose $u\neq y_i$. Then $m_i\ra u$ in $E(D)$. If $u \ra v$ in $E(D)$, then clearly \nd. Else $(u,v)\in E(D')-E(D)$. Then either $(u,v)$ is a convenient orientation w.r.t. $D$ and hence \nd or \lose . $\exists j$ such that $rs, uv\in E[Y_j,M_j]$. Assume $m_i=r$. Then $y_i=s$, $u=y_{i+1}$ and $v=m_{i+1}$. Contradiction, because $v\neq m_{i+1}$. So $r\neq m_i$. Assume $s=m_i$. Then $u=m_{i+1}$. However $(m_{i+1},m_i)\in E(D))$, then $(u,m_i)\in E(D)$, a contradiction. So $s\neq m_i$. Now we prove that $m_is$ is not a missing edge. If $v\in Y_j$, then $s\in M_j$. Then $m_is$ is not a missing edge. Else $v\notin Y_j$. Whence, $u\in Y_j$ and $s\in Y_j$. Since $m_iu$ is not a missing edge, then by definition of $G$, $fs=m_is$ is also not missing edge. But $f\ra u$ in $D$ and $u\notin N^{++}_D(s)$, then we must have $f\ra s$ in $D$. Thus $f\ra s \ra v$ in $D$.\\

  subcase b: Suppose $m_i\ra y\ra v$ in $T'$ with $u\neq y_i$ and $v\neq m_{i+1}$. Then $(m_i,u)\in E(D)$ and $(u,v)\in E(T)$.\\

  subcase b.1: Suppose $(u,v)\in E(D')$. This is the same as case a.2.2.\\

  subcase b.2: Suppose $(u,v)\in E(T)-E(D')$. Then $(u,v)$ is a convenient orientation w.r.t. $D'$. Then \nnd . Then there is a vertex $w$ such that $m_i\ra w\ra v$ in $D'$. Again this is case a.2.2.\\

  \noindent \textbf{case 2.2.2:} Assume $i=k$, that is $f=m_k$. Reorient all the missing edges incident to $f$ towards $f$. \same . Then $f$ gains no new out-neighbor. We prove that $f$ gains no new second out-neighbor. Suppose $f\ra u\ra v\ra v$ in $T'$. Then $(f,u)\in E(D)$ and $(u,v)\in E(T)$.\\

  subcase a: If $(u,v)\in E(D)$, then clearly \nd .\\

  subcase b: Suppose $(u,v)\in E(D')-E(D)$. Either $(u,v)$ is a convenient orientation (w.r.t. $D$) and hence \nd or there is \lose . There is $j$ such that $rs, uv\in E[Y_j, M_j]$. Since $f=m_k$, we have $r\neq m_k$and $s\neq m_k$. If $v\in Y_j$, then $s\in M_j$. Then $m_ks$ is not a missing edge. Else $v\notin Y_j$. Whence, $u\in Y_j$ and $s\in Y_j$. Since $m_iu$ is not a missing edge, then by definition of $G$, $fs=m_ks$ is also not missing edge. But $f\ra u$ in $D$ and $u\notin N^{++}_D(s)$, then we must have $f\ra s$ in $D$. Thus $f\ra s \ra v$ in $D$.\\

  subcase c: Suppose $(u,v)\in E(T)-E(D')$. Then $uv$ is a missing edge of $D'$ and $(u,v)$ is a convenient orientation w.r.t. $D'$. Then \nnd . Then there is a vertex $w$ such that $m_k\ra w\ra v$ in $D'$. Since $(m_k,u)\in E(D)$, then $u\neq y_k$ and $\forall j>t, u\notin Y_j$.

  Assume $w=y_k$. Then $(y_k,v)\in E(D')$. Note that $v\neq m_k$. Then $(y_k,v)\in E(D)$. Then $y_kv$ is not a missing edge. $v\notin A\cup X\cup Y$. Then either $v$ is a whole vertex or $v\in M$. If $v$ is whole, then $uv$ is not a missing edge, a contradiction. So $v\in M$. Whence, $\exists \alpha$ such that $v\in M_{\alpha}$. If $\alpha < t$, then by definition of $G$, $y_kv\in E(G)$, that is $y_kv$ is a missing edge, a contradiction. So $\alpha\geq t$. Since $v\in M_{\alpha}$ with $\alpha\geq t$ and $u\notin Y_j$ for all $j>t$, then $vu$ is not a missing edge of $D'$ (by definition of $G$). A contradiction.

  So $w\neq y_k$. Then $(m_k, w)\in E(D)$. But this is treated in case a and case b.\\

   \noindent \textbf{case 3:} $\exists 1\leq t\leq l$ such that $f\in Y_t$. \maxp $f=y_i$.\\

  \noindent \textbf{case 3.1:} Assume $(m_i,y_i)\in E(D')$. Reorient all the missing edges incident to $f$ towards $f$. \same . Clearly, $f$ gains no new first out-neighbor. We prove $f$ gains no new second out-neighbor. Assume $y_i\ra u\ra v\ra y_i$ in $T'$. Then $(y_i, u)\in E(D)$ and $(u,v)in E(T)$.\\

    subcase a: If $(u,v)\in E(D)$, then clearly \nd .\\

    Subcase b: If $(u,v)\in E(D')-E(D)$. Then either $(u,v)$ is a convenient orientation w.r.t. $D$ and hence \nd or \lose . $\exists j$ such that $rs, uv\in E[Y_j,M_j]$. Assume $r=y_i$. Then $s=m_i$, $v=y_{i+1}$ and $u=m_{i+1}$. Since $(m_i,y_i)\in E(D')$, then $(y_{i+1}, m_{i+1})\in E(D')$, that is $(v,y)\in E(D')$, a contradiction. Then $r\neq y_i$. Assume $s=y_i$. Then $u=y_{i+1}$. Hence $y_iu=y_iy_{i+1}$ is a missing edge, contradiction. So $s\neq y_i$. Now we prove that $y_is$ is not a missing edge. If $s\in Y_j$, then $u\in Y_j$, whence $y_iu$ is a missing edge, a contradiction. So $s\notin Y_j$. Whence $s\in M_j$ and $u\in M_j$. Since $y_iu$ is not a missing edge, then by definition of $G$, $y_is$ is also not a missing edge. Since $y_i\ra u$ in $D$ and $u\notin N^{++}(s)$, then we must have $s\ra y_i$ in $D$. Therefore, $y_i\ra s\ra v$ in $D$.\\

    subcase c: Suppose $(u,v)\in E(T)-E(D')$. Then $(u,v)$ is a convenient orientation w.r.t. $D'$. Whence, \nnd. Since $f\ra u$ in  $D$,then there is a vertex $w$ such that $f=y_i\ra w\ra v$ in $D'$. Then $(y_i,w)\in D$. This case is already treated in case b.

    \noindent \textbf{case 3.2:} Assume $(y_i,m_i)\in E(D')$.\\

    \noindent \textbf{case 3.2.1:} Assume $i=k$ that is $f=m_k$. Reorient all the missing edges incident to $f$ towards $f$. \same . Clearly $f$ gains no new out-neighbor. We prove that $f$ gains no new second out-neighbor. Suppose $f=y_k\ra u\ra v\ra f$ in $T'$. Then $(f,u)\in E(D)$ and $(u,v)\in E(T)$. \\

    subcase a: If $(u,v)\in E(D)$, then clearly \nd .\\

    subcase b: Suppose $(u,v)\in E(D')-E(D)$. Then either $(u,v)$ is a convenient orientation w.r.t. $D$ and hence \nd or \lose . $\exists j$ such that $rs, uv\in E[Y_j,M_j]$. Since $f=y_k$, then $r\neq y_k$ and $s\neq y_k$. Since $(y_k,u)\in E(D)$, then $u\notin Y$, so $u\in M$ and $s\in M$. Since $y_ku$ is not a missing edge, then by definition of $G$, $y_ks$ is also not a missing edge. Since $f\ra u $ in $D$ and $u\notin N^{++}_D(s)$, then we must have $f\ra s$ in $D$. Thus, $f\ra s\ra v$ in $D$.\\

    subcase c: Suppose $(u,v)\in E(T)-E(D')$. Then it is a convenient orientation w.r.t. $D'$. But $f\ra u$ in $D$ and thus in $D'$, then \nnd. so there is a vertex $w$ such that $f=y_k\ra w\ra v$ in $D'$. Since $y_ku$ is not a missing edge of $D$, then $u$ is a whole vertex of $D$ or $u\in M-\{m_k\}$. Since $(u,v)\in E(T)-E(D)$, then $u$ is not whole. Thus $\exists j$ such that $u\in M_j-\{m_k\}$. Since $f=y_k\in M_t$ and $u\in M_j-\{m_k\}$ and $y_ku$ is not missing edge, then we must have $j>t$, by using the definition of $G$. Then $\exists \alpha >j$ such that $v\in Y_{\alpha}$. Since $\alpha>j$ and $vu\in E[Y_{\alpha},M_j]$, then we must have $G[Y_{\alpha \cup M_j}]$ is a complete split graph, by using the definition of $G$. But $j>t$, then also $G[Y_{\alpha \cup M_t}]$  is a complete split graph. In particular $m_kv$ is a missing edge of $D$ and $D'$. But $(w,v)\in E(D')$, then $w\neq m_k$. Whence $(y_k,w)\in E(D)$. Since $(w,v)\in E(D')$, then this is already discussed in case $a$ and case $b$.\\

    \noindent \textbf{case 3.2.2:} Assume $i<k$. Reorient all the missing edges incident to $f$ towards $f$, except $y_im_i$. \same . Clearly $f=y_i$ gains only $m_i$ as an out-neighbor. We prove that $f$ gains only $y_{i+1}$ as a second out-neighbor. Suppose $f=y_i\ra u\ra v\ra f$ in $T'$ with $v\neq y_{i+1}$.\\

    subcase a: Assume $u=m_i$, that is $y_i\ra m_i\ra v\ra y_i$ in $T'$. Then $(y_i, m_i)\in E(D')$ and $(m_i,v)\in E(T)$.\\

    subcase a.1: Suppose $(m_i,v)\in E(D)$. Since $v\neq y-{i+1}$ and $m_iv$ is not a missing edge of $D$, then by definition of $G$, also $m_{i+1}v$ is not a missing edge. Since $m_i\ra v$ in $D$ and $m_{i+1}\notin N^{++}(m_i)$, then we must have $m_{i+1}\ra v$ in $D$. Thus $y_i\ra m_{i+1}\ra v$ in $D$.\\

    subcase a.2: Suppose $(m_i, v)\in E(D')-E(D)$. Then $m_iv=m_iy_i$ and hence, $v=y_i$, contradiction. SO this case does not exist.\\

    subcase a.3: Suppose $(m_i,v)\in E(T)-E(D)$. Then $(m_i,v)$ is a convenient orientation w.r.t. to $D'$. Since $y_i\ra m_i$ in $D'$, then $v\in N^{++}_D(y_i)$. Then there is a vertex $w$ such that $y_i\ra w\ra v $ in $D'$. If $w=m_i$, then this is already discussed in subcase a.1 and subcase a.2. Else, $w=\neq m_i$. Then $(y_i,w)\in E(D)$ and $(w,v)\in E(D')$.\\

    subcase a.3.1: If $(w,v)\in E(D)$, then $y_i\ra w\ra v $ in $D$.\\

    subcase a.3.2: Suppose $(w,v)\in E(D')-E(D)$. Then either $(w,v)$ is a convenient orientation w.r.t. $D$ and hence \nd or or \lose . If $rs=m_iy_i$, then $wv=m_{i+1}y_{i+1}$. But $v\neq y_{i+1}$, then $v=m_{i+1}$. Since $(y_i,m_i)\in E(D')$, then $(m_{i+1},y_{i+1})\in E(D')$, that is $(v,w)\in E(D)$, which is a contradiction.

    Assume $y_is$ is a missing edge of $D$. Then $s\in Y$ and hence $w\in Y$. Then $y_iw$ is a missing edge. Then $(y_i,w)\notin E(D)$, a contradiction. So $y_is$ is not a missing edge of $D$. Since $y_i\ra w$ in  $D$ and $w\notin N^{++}_D(s)$, then we must have $y_i\ra s$ in $D$. Thus $y_i\ra s\ra v$ in $D$.\\

    subcase b: Assume $u\neq m_i$. Then $(y_i,u)\in E(D)$ and $(u,v)\in E(T)$. If $(u,v)\in E(D')$, then this is already treated in subcase a.3.1 and subcase a.3.2. If $(u,v)\in E(T)-E(D')$, then it is a convenient orientation w.r.t. $D'$ and hence $v\in N^{++}_{D'}(y_i)$. But this is already treated in case a.\\

    \noindent \textbf{case 4:} $f=y\in Y_{l+1}$. Reorient all the missing edges incident to $f$ towards $f$. \same .  Note that $f$ gains no new out-neighbor. We prove it gains no new second out-neighbor. Suppose $f\ra u\ra v\ra f$ in $T'$. Then $(f,u)\in E(D)$ and $(u,v)\in E(T)$. We have the following cases.\\

    subcase a: If $(u,v) \in E(D)$, then clearly \nd .\\

    subcase b: Suppose that $(u,v)\in E(D')-E(D)$. Then either $(u,v)$ is a convenient orientation w.r.t. $D$ and hence \nd or \lose . $\exists j$ such that $rs, uv\in E[Y_j,M_j]$. Since $yu$ is not a missing edge of $D$, then $u\notin Y$. Hence $u\in M_j$ and thus $v\in Y_j$ and $s\in M_j$. Since $u,s\in M_j$ and $yu$ is not a missing edge, then also $ys$ is not a missing edge. Since $f=y\ra u $ in $D$ and $u\notin N^{++}_D(s)$, then we must have $f\ra s$ in $D$. Thus, $f\ra s\ra v$ in $D$.\\

    Subcase c: Assume that $(u,v)\in E(T)-E(D)$. Then $(u,v)$ is a convenient orientation w.r.t. $D'$. But $y\ra u$ in $D$ and $D'$, then $v\in N^{++}_{D'}(y)$. So there is a vertex $x$ such that $f\ra w\ra v$ in $D'$. Then $(f,w)\in E(D)$. But this is already treated in case $a$ and case $b$.\\

    \noindent \textbf{case 5:} $f=y\in Y_{l+2}$. Exactly same as case 4, with only one difference in subcase b. The difference is that in subcase 5.b we have $ys$ is a missing edge because there $E[Y_{l+2}, M_j]=\phi$ for sure, while in subcase 4.b we had to prove it.\\

     \noindent \textbf{case 6:} $f\in V(G)-(Y\cup M)=A\cup(X-X_1)=A\cup (X-Y1)$. Reorient all the missing edges incident to $f$ towards $f$. \same . Note that $f$ gains no new out-neighbor. We prove it gains no new second out-neighbor. Suppose $f\ra u\ra v\ra f$ in $T'$. Then $(f,u)\in E(D)$ and $(u,v)\in E(T)$. We have the following cases.\\ We have the following subcases.\\

     subcase a: If $(u,v)\in E(D)$, then clearly \nd .\\

     subcase b: Suppose $(u,v)\in E(D')-E(D)$. Suppose that $(u,v)\in E(D')-E(D)$. Then either $(u,v)$ is a convenient orientation w.r.t. $D$ and hence \nd or \lose . $\exists j$ such that $rs, uv\in E[Y_j,M_j]$.

       If $u\in Y_j$, then $fu$ is a missing edge of $D$, a contradiction. So $u\in M_j$. whence $s\in M_j$. Then $fs$ is not a missing edge. Since $f\ra u $ in $D$ and $u\notin N^{++}_D(s)$, then we must have $f\ra s$ in $D$. Thus, $f\ra s\ra v$ in $D$.\\

    subcase c: Suppose $(u,v)\in E(T)-E(D')$. Then $(u,v)$ is a convenient orientation w.r.t. $D'$. But $f\ra u$ in $D$ and $D'$, then $v\in N^{++}_{D'}(f)$. Then there is a vertex $w$ such that $f\ra w\ra v$ in $D'$. Then $(f,w)\in E(D)$ and $(w,v)\in E(D')$. But this is already discussed in subcase a and subcase.\\

    Therefore, in all cases $f$ has the SNP in $D$.

\end{proof}

\begin{cor}\cite{a}\label{SNCgs}
 Every oriented graph missing a threshold satisfies SNC.
\end{cor}

\begin{cor}
 Every oriented comb satisfies SNC.
\end{cor}

\begin{proof}
  Let $D$ be an orientation of a comb $G$. Then the missing graph $G'$ of $D$ is either $\overline{G}$ or $\overline{G}$ without its isolated vertices. In both cases, $G'$ is again a comb. Thus, by the previous theorem, $D$ satisfies SNC
\end{proof}

\begin{cor}
 Every oriented threshold graph satisfies SNC.
\end{cor}

 We end by the following problem and question:

\begin{prob}
  Does every oriented graph missing a graph in $\mathcal{F}(\vec{\mathcal{P}})$ satisfies SNC?
\end{prob}

\end{section}

\end{document}